\documentclass[12pt,BCOR5mm]{amsart}

\usepackage{fontspec}

\usepackage{stackrel}
\usepackage[mathscr]{euscript}
\let\euscr\mathscr \let\mathscr\relax
\usepackage{array}
\usepackage[scr]{rsfso}
\usepackage{dashrule}
\usepackage{ytableau}
\usepackage{young}
\usepackage{youngtab}

\usepackage{rotating}

\usepackage{stmaryrd}
\usepackage{amssymb}
\usepackage{amsmath}
\usepackage{titletoc}
\usepackage{float}

\usepackage{extarrows}
\usepackage{varioref}
\usepackage{bm}
\usepackage{amsthm}
\usepackage[all,cmtip]{xy}
\usepackage{tikz-cd}
\usepackage{color} 
\usepackage{soul}
\setulcolor{red}
\sethlcolor{yellow}
\setstcolor{blue}
\usepackage{amscd}
\usepackage{xfrac}    

\usepackage{amsfonts}
\usepackage[colorlinks,linkcolor=blue,citecolor=blue, pdfstartview=FitH]{hyperref}
\usepackage{graphicx}

\usepackage{tcolorbox}
\usepackage{colortbl}
\tcbuselibrary{breakable}
\usepackage{cite}

\usepackage{geometry}
\numberwithin{equation}{section}

\theoremstyle{definition}
\newtheorem{thm}{Theorem}[section]
\newtheorem{theorem}[thm]{Theorem}

\newtheorem{lemma}[thm]{Lemma}
\newtheorem{corollary}[thm]{Corollary}
\newtheorem{proposition}[thm]{Proposition}

\newtheorem{remark}[thm]{Remark}

\newtheorem{definition}[thm]{Definition}

\newtheorem{defn-thm}[thm]{Definition-Theorem}

\title{Frobenius induced morphisms on moduli of sheaves on curves}
\author{Jin Cao}
\address{Jin Cao: School of Mathematics and Physics, University of Science and Technology Beijing, Beijing 100083, P. R. China.}
\email{caojin@ustb.edu.cn}
\author{Xiaoyu Su}
\address{Xiaoyu Su: School of Mathematical Sciences, Beijing University of Posts and Telecommunications, Beijing 100876, P. R. China;
Key Laboratory of Mathematics and Information Networks (Beijing University of Posts and Telecommunications), Ministry of Education, China.}
\email{xiaoyusu@bupt.edu.cn}

\begin{document}
	\maketitle
    \begin{abstract} In this paper, we show that the Frobenius pullback of a general semi-stable vector bundle in the moduli space of vector bundles with fixed rank and degree is still semi-stable by deformation trick. We then present several applications of the main theorem.  \\

    \end{abstract}

	\section{Introduction}

The Frobenius morphism provides deep insights into the geometry of moduli spaces—such as those of vector bundles and Higgs bundles—on curves. A central and well-studied problem is the following:

    \textbf{
What is the action of the Frobenius pushforward $\mathrm {fr}_*$
and pullback $\mathrm {fr}^*$ on the corresponding moduli spaces? More specifically, one may ask: how is (semi)stability affected by iteration of the Frobenius morphism?
    }
	
	To address these problems, we prove the following theorem.
	\begin{theorem}[Theorem \ref{Thm:main1}]  Let $X$ be a smooth projective curve of genus $g\geq 2$ over an algebraically closed field $k$ of characteristic $p > 0$, and $\mathrm {fr} : X \to X^{(1)}$ the relative Frobenius morphism. 
		\begin{itemize}
			\item[(a)] the Frobenius pullback $\mathrm {fr}^*E$ of a general semi-stable vector bundle $E$ with rank $r$ and degree $d$ in $\mathrm {Bun}^\text{ss}_{(r,d)}$ is still semi-stable.  
			\item[(b)] the set-theoretic map \[\phi^\mathrm {Bun}_{(r,d)}:\mathrm {Bun}_{X^{(1)},(r,d)}^\text{ss} \dashrightarrow \mathrm {Bun}^\text{ss}_{X,(r,pd)},\ E\mapsto \mathrm {fr}^*E\] inducecd by Frobenius pullback is a dominant rational map on the moduli spaces.
		\end{itemize}
	\end{theorem}

	The study of semistability under pushforward and pullback by finite morphisms is a rich topic in both moduli theory and vector bundle theory.  For finite separable maps between normal varieties, it is well known (cf. \cite[Lemma 1]{Gieseker79} and \cite[Lemma 3.2.2]{HL10}) that, if $f : V\to W$ is a finite separable map between normal varieties, then a torsion free sheaf $F\in \mathrm {Coh}(W)$ is slope $H$-semistable if and only if $f^*F$ is slope $f^*H$-semistable. In the positive characteristic case, we have the relative Frobenius map $\mathrm {fr}:X\to X^{(1)}$ which is purely inseparable and it is interesting to investigate the semistability of $\mathrm {fr}_*E$ and $\mathrm {fr}^*E$ for a semistable vector bundle $E$ on $X$. More precisely, for the Frobenius pushforward, the stability of $\mathrm {fr}_*L$ with $L$ being a line bundle on $X$ was first proven by H. Lange and C. Pauly in \cite[Proposition 1.2]{LangePauly08}).  In \cite{MP07}, V. Mehta and C. Pauly showed that for a curve $X$ of genus $g \geq 2$, if $E$ is semistable, then $\mathrm {fr}_*E$ is also semistable by the covering trick together with G. Faltings's semi-stability criterion (cf. \cite[Theorem I.2]{Fal93} and \cite[Lemme 2.1,  Th\'eor\`eme 2.4 and Lemme 2.5]{LeP96}). Using a clever direct computation, X. Sun \cite[Theorem 2.2]{Sun08} showed that if $E$ is semistable, then $\mathrm{fr}_*E$ is also semistable. Moreover, he also considered the higher-dimensional case and gave a criterion of the instability of the Frobenius direct image sheaf (cf. \cite[Theorem 4.8]{Sun08}). 
    
    For the Frobenius pullback, D. Gieseker in \cite{Gieseker73} showed that  for each prime $p$ and  any integer $g \geq 2$, there is a curve $X$ of genus $g$ defined over a field of characteristic $p$ and a semi-stable bundle $E$ of rank two on $X$ so that $\mathrm {fr}^*E$ is not semi-stable (see also \cite{LangePauly08,Os08}). Moreover, if we assume the degree of the bundle is zero, V. B. Mehta and S. Subramanian in \cite{MS95} proved that for any ordinary curve, the Frobenius pullback induces a dominate rational map $\phi^\mathrm {Bun}:\mathrm {Bun}_{(r,0)}^\text{ss} \dashrightarrow \mathrm {Bun}^\text{ss}_{(r,0)}$. Based on an unpublished work of J. de Jong (cf. \cite[Appendix A, Theorem 6]{Os06}) , B. Olsserman, J. de Jong and C. Pauly dropped the assumption of ordinarity and  showed that the relative Frobenius map $\mathrm {fr}:X\to X^{(1)}$ induces (by the pull-back) a rational map between the moduli space of bundles with degree zero over a curve of genus $g \geq  2$. 
\begin{remark}
In the first version of this work, we mistakenly claimed that there are only finitely many isomorphism classes of bundles whose Frobenius pullback yields an isomorphic bundle. The editor kindly pointed out several counterexamples to this statement—for instance, in \cite[Proposition 4.1]{Pauly07}, Pauly provided an explicit example of such a family.  
\end{remark}

The rough idea of the proof of our theorem is as follows:
\begin{itemize}
\item 
Given any unstable vector bundle $E$ of rank $r$ and degree $d$, there is a deformation from a general bundle to $E$ in $\mathcal Bun_{(X,r)}$, i.e., we may choose a smooth curve $T$ together with a map $T \to \mathcal Bun_{(X,r)}$, such that the corresponding flat family of bundles $\mathcal E$ satisfies $\mathcal E|_{\{t_0\}}\cong E$ and for $t\in T-\{t_0\}$, $\mathcal E|_t$ can be chosen general.
\item
Next we show that the Frobenius pullback bundle $\mathrm {fr}^*E$ is semi-stable for general $E$. The proof proceeds by contradiction. Assume that  $\mathrm {fr}^*E$ is unstable for general $E\in \mathcal Bun_{\mathrm {SL}_r,d_\ell}^\text{ss}$. Then after iterative Frobenius pullback, we show that $\mathrm{fr}^{(n)*} E^{\mathrm{univ}}$ splits as a direct sum of proper sub bundles of a fixed type. However, this imposes a restriction on the possible deformations from a general bundle $\mathrm{fr}^{(n)*} E$ to unstable ones, which would then produce a contradiction.
\item 
The final step is to demonstrate the stability of $\mathrm {fr}^*E$ for general $E$ and show the main theorem. 
\end{itemize}


The structure of the paper is as follows. Section 2 introduces the preliminary notions and conventions used throughout. The detailed proofs of the aforementioned claims are provided in Section 3. In Section 4, we explore various applications of our main result; among these, we offer a novel proof that the Frobenius morphism preserves semistability in the case of vector bundles on curves.

\section{Preliminaries}

One of the most striking differences between algebraic geometry in characteristic zero and positive characteristic is the existence of the Frobenius morphism. This map is a fundamental tool in positive characteristic geometry and number theory, but without any direct analogue in characteristic zero. Let $X$ be a variety defined over an algebraically closed field $k$ with characteristic $p > 0$. The absolute Frobenius morphism $ \mathrm {Frob}_X: X \to X $ is the identity map on the underlying topological spaces and on the structure sheaf defined by raising functions to the $p$-th power: $\mathrm {Frob}_X^\flat: \mathcal{O}_X \to \mathrm{id}_{X*} \mathcal{O}_X, \quad f \mapsto f^p$. This is a $\mathbb F_p$-linear endomorphism of $X$ (since every element in $\mathbb F_p$ is fixed by taking $p$-th power), but it is not a $ k $-morphism because it permutes the $p^n$-th roots of unity for $n\geq 2$. To obtain a morphism over $k$, we decompose $\mathrm {Frob}_X$ with the base change induced by the Frobenius on $k$ and obtain the commutative diagram of Frobenii:
	$$
	\begin{tikzcd}
		X \arrow[rdd, bend right] \arrow[rrd, bend left,"{\mathrm {Frob}_X}"] \arrow[rd,"{\mathrm {Frob}_{X/k}}"] &                                  &             \\
		& X^{(1)} \arrow[d] \arrow[r,"{(\mathrm {Frob}_{k})_X}"] \arrow[rd,phantom, "\square"] & X \arrow[d] \\
		& \mathrm {Spec}(k) \arrow[r,"{\mathrm {Frob}_k}"']                      & \mathrm {Spec}(k)       
	\end{tikzcd}$$
	In the diagram above, $\mathrm {Frob}_{X}$ denotes the absolute Frobenius map and $\mathrm {Frob}_{X/k}: X\to X^{(1)}$ denotes the relative Frobenius map.  We will denote the relative Frobenius $\mathrm {Frob}_{X/k}$ by $\mathrm {fr}$ shortly.

    In the rest of this paper, we let $X$ be a smooth projective curve of genus $g \geq 2$ over an algebraic closed field $k$ of characteristic $p>0$ and we will consider various types of moduli spaces over $X$ such as sheaves (especially vector bundles), Higgs bundles, flat connections and so on. For those moduli spaces, we denote the corresponding moduli stack by $\mathcal  Bun_{X,(r,d)}$, ${(\mathcal  M_\text{Dol})}_{X,(r,d)}$ and ${(\mathcal  M_\text{dR})}_{X,(r,d)}$ respectively, where $r$ is the rank of the underlying bundles and $d$ is the degree of the bundles. If the corresponding data $(X,r,d)$ is clear, we will omit it and just denote them by  $\mathcal  Bun$, ${\mathcal  M_\text{Dol}}$ and $\mathcal  M_\text{dR}$. We use the superscript $(\_)^\text{s}$ (resp. $(\_)^\text{ss}$) to denote the open locus of stable (resp. semistable) objects. We denote by $\mathrm {Bun}$, $M_\text{Dol}$ and   $M_\text{dR}$ the corresponding moduli spaces corepresent the functors of semistable objects up to S-equivalences (cf. \cite[Theorem 1.1]{Lan14}) and denote their open subsets of stable objects by $\mathrm {Bun}^s$, $M_\text{Dol}^\text{s}$ and $M_\text{dR}^\text{s}$.

	Let $E$ be a torsion free coherent sheaf on $X$, a $k$-connection $\nabla$ is a $k$-linear map $\nabla:E\to E\otimes\Omega_X^1$ satisfies the Leibniz rule 
	$$\nabla(fe)=e\otimes df+f\nabla e,$$
	where $f$ is a local section of $\mathcal O_X$ and $e$ a local section of $E$. By evaluating the differentials with tangent vector $D$ by $\nabla_D(-)=\langle\nabla(-),D\rangle$,  a $k$ connection  is equivalent to a map of $\mathcal O_X$ modules $$\euscr T_X\to \euscr E nd_k(E),\ D\mapsto \nabla_D$$ and we still denote this map by $\nabla$. Moreover, if $\nabla\circ \nabla=0$, then we call it a flat connection or an integrable connection. On a curve, all connections are flat. If $\nabla$ is flat then $D\mapsto \nabla_D$ from the tangent sheaf $\euscr T_X$ to $\euscr End(E)$ extends to the universal enveloping algebra $\mathtt U(\euscr T_X)=\euscr D_X$ as a map of $\mathcal O_X$-algebra $\euscr D_X\to \euscr End(E)$. The local sections of the universal enveloping algebra $\euscr D_X$ are called the crystalline differential operators.  Since the characteristic of $k$ is $p>0$, then $\nabla_{D}( - )$ is $p$-linear. Hence it is equivalent to a map 
	$$\euscr T_{X/k}\to \euscr E nd_{\mathrm {fr}^{-1}\mathcal O_{X^{(1)}}}(E).$$
	Then one can define its $p$-curvature map (which is a map of sheaves) as:
	$$\psi_\nabla: \euscr T_{X/k}\to \euscr E nd_{\mathcal O_{X}}(E),\ D\mapsto [(\nabla_{D})^p-\nabla_{D^p}].$$
	Here, one can check that $\psi_\nabla(D):E\to E$ is $\mathcal O_X$-linear for any local section $D$ of the tangent sheaf $\euscr T_{X/k}$. In fact, the $p$-curvature map $\psi_\nabla$ is $p$-linear. i.e., it is additive, and $\psi_{\nabla}(gD)=g^p\psi_{\nabla}(D)$ for any local section $g$ of $\mathcal O_X$. This is not straightforward to check and one can find a proof in \cite[Proposition (5.2)]{Ka70} or a different proof in \cite[1.3, 2.1.2 Remark]{BMR08}. 
	
	According to the $p$-linear property, the sheaf map $\psi_\nabla$ can be viewed as a map of $\mathcal O_{X^{(1)}}$ modules $\euscr T_{X^{(1)}/k}\to \mathrm {fr}_*\euscr End_{\mathcal O_X}(E)$ which adjoint to a morphism $\mathrm {fr}^*\euscr T_{X^{(1)}/k}\to\euscr End_{\mathcal O_X}(E)$ of $\mathcal O_{X}$ modules. 
	In what follows, we use $\psi_\nabla$ to denote the $\mathcal O_X$ linear map  $\mathrm {fr}^*\euscr T_{X^{(1)}/k}\to\euscr End_{\mathcal O_X}(E)$. Moreover, one can also check that for any local sections $D$, $D'$ of $\euscr T_{X^{(1)}/k}$, $\psi_\nabla(D)$ and $\psi_\nabla(D')$ commute as elements in $\euscr End_{\mathcal O_X}(E)$. Thus, if the tangent sheaf is reflexive, we can regard the $p$-curvature as a $\mathrm {fr}^*\Omega_{X^{(1)}}^1$ valued Higgs field on $E$. In \cite[Theorem 1.1]{Lan14}, A. Langer constructed the moduli space $M_\text{dR}^\text{ss}$ which parametrizes the S-equivalent classes of semi-stable flat connection $(E,\nabla)$ with $\mathrm{rank}(E)=r$ and $\deg(E)=d$.

	\section{Stablity of Frobenius pushforward or pullback of bundles}

	Let us first consider the Frobenius pull back map on the level of stacks. For a $k$ scheme $T$, we denote the base change of the relative Frobenius map $\mathrm {fr}\times \mathrm {id}_T$ by $\mathrm {fr}_T$ and we have a functor
	$$ {\phi}^{\mathcal Bun}_{T}:\mathcal Bun_{(X^{(1)},r)}(T)\to {\mathcal Bun}_{(X,r)}(T),\quad E/X^{(1)}\times T\mapsto \mathrm {fr}_T^*E,$$ which is compatible with the pull back morphisms. Thus we get a morphism of algebraic stacks $\phi^{\mathcal Bun}:\mathcal Bun_{(X^{(1)},r)}\to \mathcal Bun_{(X,r)}$ and we will abbreviate it by $\phi$. The notation follows form \cite{Ne05}.

	Frobenius pullback bundles possess a natural flat connection. For any $k$ scheme $T$, if $E$ is a coherent sheaf on $X^{(1)}\times_kT$ then one can construct a canonical connection $\nabla^{\text {can }}=\mathrm d_{X\times T/T}\otimes_{\mathcal O_{X^{(1)}\times T}} \mathrm {id}_E$ on $\mathrm {fr}_T^* E$. One can check that the $p$-curvature of a canonical connection is always zero. Indeed, we have the following Cartier descent theorem:
		\begin{proposition}[Theorem (5.1) \cite{Ka70}] Let $X/k$ be a smooth curve of characteristic $p > 0$. For any $k$ scheme $T$, denote by $\mathrm{fr}_{T}: X\times {T} \rightarrow X^{(1)}\times{T}$ the relative Frobenius $\mathrm {Frob}_{X\times T/T}$. Then there is an equivalence of categories between the category of the flat family of vector bundles on $X^{(1)}\times T$ and the full subcategory of relative flat connections consisting of objects $( F, \nabla )$ whose $p$-curvature is zero. 
		\end{proposition}
        \begin{remark}
        This equivalence may be given explicitly as follows: Let $E$ be a quasi coherent sheaf on $X^{(1)}\times T$. Then there is a unique integrable $T$-connection $\nabla^{\text {can }}$ with zero $p$-curvature on $\mathrm{fr}_T^*(E)$, such that
		$E \xrightarrow{\cong}\mathrm{Ker}(\nabla^{\text{can}}).$
		The desired functor is $E \mapsto\left(\mathrm{fr}_T^*(E), \nabla^{\text {can }}\right)$.
		Consider a relative flat connection $(V, \nabla)$ on the family $X\times T / {T}$. Assuming its $p$-curvature is zero, we form ${V}^{\nabla}=\mathrm{Ker}(\nabla)$, which is in a natural way a coherent sheaf on $X^{(1)}\times T$. The desired inverse functor is given by $({V}, \nabla)\mapsto \mathrm {Ker}(\nabla).$
        \end{remark}
		Moreover, A. Langer in \cite[Proposition 2.2]{Lan04positive} showed that the previous functor is compatible with semi-stability conditions. That is, a coherent sheaf $E$ on $X^{(1)}\times T$ is semistable if and only if the sheaf $\mathrm {fr}_T^* E$ is $\nabla^{\text{can}}$-semistable. 
	
		The framework of stacks provides a natural setting for Cartier descent. Let 
		$$\mathcal M_{\text{dR},r}=\bigsqcup_{\bar d\in \mathbb Z}\mathcal M_{\text{dR},(r,\bar dp)}$$
		be the stack of flat connections on $X$ such that for any $k$ scheme $T$, $\mathcal M_{\text{dR} ,r}(T)$ is the groupoid of flat $T$-connections. By \cite{BS06}, a flat connection on a curve is always of degree divided by $p$, thus the union is taking along $\bar dp,\ \bar d\in \mathbb Z$. We denote by $\mathcal M_\text{dR}^{\psi=0}$ the closed substack of $\mathcal M_\text{dR}$ parametrizing connections with vanishing $p$-curvature(cf. \cite[4]{LP01}, \cite[Page 531]{Lan14}). Then the previous Frobenius pullback defines a functor
		$$\mathrm {F}^\mathcal B:\mathcal B un_{(X^{(1)},r)}\to \mathcal M_{\text{dR},(X,r)},\quad E\mapsto (\mathrm {fr}^*E,\nabla^\text{can}),$$
		which factors as an equivalence followed by a closed immersion:
		$$\begin{tikzcd}
				{\mathcal Bun_{(X^{(1)},r)}} \arrow[rr, "\mathrm {F}^\mathcal B"] \arrow[rd, "\cong"'] &                                                           & {\mathcal M_{\text{dR},(X,r)}} \\
				& {\mathcal M_{\text{dR},(X,r)}^{\psi=0}} \arrow[ru, hook] &                               
			\end{tikzcd}.$$
		For any flat connection $(V,\nabla)$, forgetting the connection $\nabla$ yields a family of vector bundles $V$. This defines a forgetful functor 
		$$\mathsf {fg}:\mathcal M_{\text{dR},(X,r)}\to \mathcal Bun_{(X,r)},\quad (V,\nabla)\mapsto V,$$
		such that the composition
		$$\phi=\mathsf {fg}\circ \mathrm {F}^\mathcal B:\mathcal Bun_{(X^{(1)},r)}\to \mathcal Bun_{(X,r)},$$ equals the functor given by Frobenius pullback,
	
		

%
%



To construct the universal family of bundles, we pass to a smooth atlas of the moduli stack of bundles. Let $\tau$ be a fixed Harder-Narasimhan polygon type and $P(X)= rX + d + r(1 − g)$ be the Hilbert polynomial. By \cite[Lemma 1.7.9]{HL10}, let $t_0$ be an integer which is sufficiently large. Consider the quote scheme $\mathrm {Quot}_{\mathcal O_{X}^{P_\ell(t_0)}(-t_0)/X/k}^{P_\ell(t)}$ with the universal quotient bundle $E^\text{univ}$ and a $\mathrm {GL}_{P_\ell(t_0)}$ action. There is an open subset $R^{\leq \tau}$ in the quote scheme (thus depends on $t_0$) satisfies
\begin{itemize}
	\item[1.]  for all $q\in R^{\leq \tau}$, $E^\text{univ}|_q$ is torsion free and $E^\text{univ}(t_0)$ is globally generated.
	\item[2.] the evaluation map $\mathrm H^0(X,\mathcal O_{X}^{P_\ell(t_0)})\to \mathrm H^0(X,E^\text{univ}|_q(t_0))$ induced by $q$ is an isomorphism, and	$\mathrm H^i(X,E^\text{univ}|_q(t_0)) = 0$ for all $i \geq 1$.
	\item[3.] for all $q\in R^{\leq \tau}$, the Harder-Narasimhan polygon of $E^\text{univ}|_q$ is below or equal than $\tau$. 
\end{itemize}
Then by \cite[Proposition 2.2.8]{HL10} and \cite[Remark 5.5]{Newstead2012}, the obstruction vanishes for all sufficiently large $t_0$. Therefore every point $q\in R^{\leq \tau}$ is smooth. In this case, $R^{\leq \tau}$ is an irreducible smooth quasi projective variety of dimension $P_\ell(t_0)^2+r^2(g-1)$ (cf. \cite[Remark 5.5]{Newstead2012} or \cite[PREMI\`ERE PARTIE, ROPOSITION 23]{Seshadri82}). In particular, for semi-stable bundles, the associated moduli space forms a smooth irreducible variety $R^\text{ss}$. By the respectability of the quote scheme, we have smooth atlas $R^{\leq \tau}\to \mathcal Bun_{r,d}$. Moreover, there is the universal family $\mathcal O_{X\times R^{\leq \tau}}^{P_\ell(t_0)}(-t_0)\twoheadrightarrow E^\text{univ}$ on $R^{\leq \tau}$. 

For the case of $\mathrm {SL}_r$ bundles, we let $\mathcal Bun_{\mathrm {SL}_r,\mathcal L}$ be the moduli stack of vector bundles with fixed determinant $\det E=\mathcal L$ and rank $r$. In particular, for the line bundle $\mathcal L=\mathcal O_X(dx_0)$, we denote the corresponding moduli stack by $\mathcal Bun_{\mathrm {SL}_r,d}$.  In this case, the determinant of the universal bundle $\det(E^\text{univ})$ defines a map $R^{\leq \tau}\to \mathrm {Pic}^d_X$ with the tangent map (see \cite[Proposition 2.2.7]{HL10}) at $[\mathcal O_{X}^{P_\ell(t_0)}(-t_0)\twoheadrightarrow E]\in R^{\leq \tau}$ given by 
$$T_R|_{[\mathcal O_{X}^{P_\ell(t_0)}(-t_0)\twoheadrightarrow E]}=\mathrm {Hom}(K,E)\xrightarrow{\delta}\mathrm {H}^1(X,\mathrm {End}(E))\xrightarrow{\mathrm {H}^1(\mathrm {tr})}   \mathrm {H}^1(X,\mathcal O_X)=T_{\mathrm {Pic}^d_X}|_{\det E},$$  
where $K$ is the kernel of $\mathcal O_{X}^{P_\ell(t_0)}(-t_0)\twoheadrightarrow E$. This map is surjective and consequently the map $R^{\leq \tau}\to \mathrm {Pic}^d_X$ is smooth. With this notation, we denote the inverse image at $\mathcal L$ by $R^{\leq \tau}_{\mathcal L}$. This is a smooth variety and the corresponding map $R_\mathcal L^{\leq \tau}\to \mathcal Bun_{\mathrm {SL}_r,\mathcal L}$ is a smooth atlas with universal family $\mathcal O_{X\times R^{\leq \tau}}^{P_\ell(t_0)}(-t_0)\twoheadrightarrow E^\text{univ}$.

Let $\mathrm {fr}^{(n)}:X\to X^{(1)}$ be the $n$-th relative Frobenius map. Let $\mathcal Bun_{\mathrm {SL}_r,d_\ell}$ be the moduli stack of vector bundles with the fixed determinant $\det E=\mathcal O_X((d+r\ell)x_0)$ of rank $r$ and degree $d_\ell=d+\ell r$.  We may assume $\ell$ sufficiently large.

\subsection{Splitting of Frobenius pull back}

In this subsection, we show that if $\mathrm {fr}^*E$ is unstable for general $E\in \mathcal Bun_{\mathrm {SL}_r,d_\ell}^\text{ss}$, then for sufficiently large $n\in \mathbb Z_+$, $\mathrm {fr}^{(n)*}E$ splits as a direct sum of proper sub bundles for general $E\in \mathcal Bun_{\mathrm {SL}_r,d_\ell}^\text{ss}$. 

Let us recall the fundamental definitions and properties of ample and semistable vector bundles.

\begin{definition}[\cite{Hart66} section 2]   A vector bundle $E$ on a scheme $X$ is ample if for every coherent sheaf $F$, there is an integer $n_0>0$, such that for every $n \geq n_0$, the sheaf ${F} \otimes \mathrm{Sym}^n({E})$ (where $\mathrm{Sym}^n(E)$ is the $n^{\text {th }}$ symmetric product of $E$) is generated as an $\mathcal{O}_{X}$-module by its global sections.
\end{definition}
	A bundle $E$ on $X$ is ample as a vector bundle if and only if the tautological line bundle $\mathcal{O}_{\mathbf{P} (E)}(1)$ is an ample line bundle on the projective bundle $\mathbf{P}({E})=\mathrm {Proj}_{X}\mathrm {Sym}^{^\bullet}(E)$ (cf. section 3 of \cite{Hart66}). Moreover, if $X$ is proper over $k$ and let
$$
0 \rightarrow {E}^{\prime} \rightarrow {E} \rightarrow {E}^{\prime \prime} \rightarrow 0
$$
be a short exact sequence of vector bundles on $X$ , with ${E}^{\prime}$ and ${E}^{\prime \prime}$ ample. Then $E$ is ample. Indeed, if we assume $X$ is normal and $\mathrm {char}(k)=p>0$, then $E$ is ample if and only if $\mathrm{fr}^*{E}$ is ample.
	
	In the case of the positive characteristic, we have the concept called cohomologically $p$-ample in \cite{Gieseker71} and called F-ample in \cite{Ara04}.

	\begin{definition} Let $X$ be a smooth projective variety defined over $k$ with $\mathrm {char}(k)=p>0$. A coherent sheaf $E$ is called cohomologically $p$-ample if for any locally free sheaf ${F}$, there exists a positive integer $N$ such that
		$$
		\mathrm H^i(X^{(-m)}, \mathrm{fr}^{(m)*}{E} \otimes {F})=0
		$$
		for all $i>0$ and $m>N$.
	\end{definition}
	By \cite[PROPOSITION 5.4]{Ara04}, for a smooth projective curve $X$ over a field $k$, a vector bundle $E$ is cohomologically $p$-ample if and only if $E$ is ample. 
	
	\vspace{10pt}
	
In positive characteristic case, there is another very important notion of the stability condition, namely the strongly semistability(see \cite[Introduction]{Lan04positive} and \cite{Lan08,Lan09}). 
	
\begin{definition} Let $X$ be a projective variety over $k$ and $H$ an ample bundle on $X$.  A coherent sheaf $E$ is called strongly slope $H$-semi-stable, if for any positive integer $\ell$, the pull back $(\mathrm {fr}^*_X)^{\ell}E$ is slope $H$-semi-stable.
\end{definition}

The torsion-free part of a tensor product of strongly slope semistable sheaves is itself strongly slope semistable (cf. \cite[COROLLARY A.3.1]{Lan04mixed}). For other properties of strongly semi-stable sheaves, we refer the readers to \cite{Lan04positive,Lan08,Lan09} and references there in. 
	
	For strongly semi-stable bundles on curves, we have the following lemma.
	
	\begin{lemma}\label{lemma:strample} Let $X$ be a smooth projective curve over a field $k$ with $\mathrm {char}(k)=p>0$. If $E$ is a strongly semi-stable vector bundle with $\deg(E)>0$, then $E$ is ample.
	\end{lemma}
	\begin{proof} According to \cite[Proposition(7.2)]{Hart66} for a fixed rank, there exists an integer $d_0$ such that every indecomposable bundle of degree $\geq d_0$ on a non-singular projective curve $X$ is ample. Let $E$ be a strongly semi-stable bundle with $\deg(E)=d>0$. Then for  $n$ sufficiently large, the indecomposable factors of $\mathrm {fr}^{(n)*}E$ have the same slopes and, consequently, sufficiently large degree. Hence for $n$ large enough, $\mathrm {fr}^{(n)*}E$ is ample and then $E$ is ample. 
	\end{proof}
	 
Indeed, the comparison between ampleness and cohomologically $p$-ampleness yields the following result.
	 
\begin{corollary}
Let $X$ be a smooth projective curve over a field $k$ with $\mathrm {char}(k)=p>0$. If $E$ is a strongly semi-stable vector bundle with $\deg(E)>0$, then for $n$ sufficiently large, $\mathrm {H}^1(X,\mathrm {fr}^{(n)*}E)=\mathrm {H}^1(X,\mathrm {fr}^{(n)*}E\otimes \mathcal O_X)=0$. 
\end{corollary}

Next we have:
\begin{lemma}\label{lemma:ustsplit} 	 
Assume that $\mathrm {fr}^*E$ is unstable for general $E\in \mathcal Bun_{\mathrm {SL}_r,d_\ell}^\text{ss}$. Then for a general $E\in \mathcal Bun_{\mathrm {SL}_r,d_\ell}^\text{ss}$, $\mathrm {fr}^{(n)*}E$ splits as a direct sum of sub bundles for sufficiently large $n\in \mathbb Z_+$. 
\end{lemma}

\begin{proof} To construct the universal family, we should consider the smooth covering $R=R_{\mathcal O(d_\ell)}^\text{ss} \to \mathcal Bun_{\mathrm {SL}_r,d_\ell}^\text{ss}$. Let $E^\text{univ}_R$ be the universal family of semi-stable bundle with rank $r$ and determinant $\mathcal O_X(d_\ell x_0)$. Assume the generic Harder-Narasimhan filtration (cf. \cite[Theorem 3.1]{Nitsure11}) $\tau_1$ of $\mathrm {fr}_R^*E^\text{univ}_R$ is non-trivial. 
	
Let $K(R)$ be the fraction field of $R$ and $\overline {K(R)}$ be its algebraic closure. According to \cite[Theorem 2.7]{Lan04positive}, there exists $N_0$ such that all quotients in the Harder-Narasimhan filtration of $\mathrm {fr}^{(N_0)*}_{\overline{K(R)}}E^\text{univ}_{\overline{K(R)}}$ are strongly semistable. Then for $n\geq N_0$, the Harder-Narasimhan filtration $\tau_n$ of $\mathrm {fr}^{(n)*}_{\overline{K(R)}}E^\text{univ}_{\overline{K(R)}}$ are just the Frobenius pull back of $\tau_{N_0}$. Moreover, the graded pieces $$\mathrm {gr}_i^{\tau_{n}}\mathrm {fr}^{(n)*}_{\overline{K(R)}}E^\text{univ}_{\overline{K(R)}}=\mathrm {fr}^{(n-N_0)*}\mathrm {gr}_i^{\tau_{N_0}}\mathrm {fr}^{(N_0)*}_{\overline{K(R)}}E^\text{univ}_{\overline{K(R)}}$$
are all strongly semi-stable. Let $i>j$, the bundle
\begin{equation}\label{eq:gr}
	\left(\mathrm {gr}_i^{\tau_{n}}\mathrm {fr}^{(n)*}_{\overline{K(R)}}E^\text{univ}_{\overline{K(R)}}\right)^\vee\otimes \left( \mathrm {gr}_j^{\tau_{n}}\mathrm {fr}^{(n)*}_{\overline{K(R)}}E^\text{univ}_{\overline{K(R)}}\right)
\end{equation}
are tensor products of strongly semi-stable bundles and thus semi-stable. By definition of the Harder-Narasimhan filtration, the bundles in (\ref{eq:gr}) are of positive slope, and hence ample by Lemma \ref{lemma:strample}. Then for $n$ sufficiently large, for any $i>j$, we have
$$\mathrm {H}^1(X_{\overline{K(R)}},	\left(\mathrm {gr}_i^{\tau_{n}}\mathrm {fr}^{(n)*}_{\overline{K(R)}}E^\text{univ}_{\overline{K(R)}}\right)^\vee\otimes \left( \mathrm {gr}_j^{\tau_{n}}\mathrm {fr}^{(n)*}_{\overline{K(R)}}E^\text{univ}_{\overline{K(R)}}\right))=0.$$
By the flat base change property, for $n\geq N_0$  sufficiently large, there are open subsets $U_n\subset R$ such that 
$$\mathrm {H}^1(X_{U_n}/U_n,	\left(\mathrm {gr}_i^{\tau_{n}}\mathrm {fr}^{(n)*}_{U_n}E^\text{univ}_{U_n}\right)^\vee\otimes \left( \mathrm {gr}_j^{\tau_{n}}\mathrm {fr}^{(n)*}_{U_n}E^\text{univ}_{U_n}\right))=0.$$
So for $n\geq N_0$  sufficiently large, the Frobenius pullback sheaf $\mathrm {fr}^{(n)*}_{U_n}E^\text{univ}_{U_n}$ splits as a direct sum of sub bundles. Since $\tau_1$ is non-trivial and $\tau_n$ are finer than $\tau_1$, then those sub bundles are proper sub bundles.
\end{proof}

\begin{remark} The splitting over a point is also proved in \cite[Proposition 2.1]{BP04}.
\end{remark}

\subsection{Canonical connections on semi-stable bundle}

In this subsection, we show that a general stable bundle $E$ of rank $r$ and degree d in $\mathcal Bun_{r,d}$, $\mathrm {fr}^*E$ is semi-stable. The proof proceeds by contradiction. Assume that for a general stable bundle $E$, $\mathrm {fr}^*E$ is unstable. By the previous Lemma \ref{lemma:ustsplit}, we choose $n\geq N$ such that the Harder Narasimhan type of the generic Frobenius pull back splits. We start with the rank $2$ case.

\begin{lemma} A general stable bundle $E$ of rank $2$ and degree $d$ in $\mathcal Bun_{2,d}$, $\mathrm {fr}^*E$ is semi-stable. 
\end{lemma}
\begin{proof} In the rank $2$ case, if the pull back of a general stable bundle is an unstable bundle, then for $N$ large and general $E$, $\mathrm {fr}^{(N)*}E$ splits as a direct sum of two line bundles $V_1\oplus V_2$ with $\deg(V_1)=d_1$, $\deg (V_2)=p^Nd-d_1$, $d_1>d_2$. Taking $N$ large enough, we may assume that $d_1-d_2>2g-2$. Consider the unstable rank $2$ bundle $L_1\oplus L_2$ with $\deg L_1=d_1+2g-2$, $\deg L_2=d-d_1-2g+2$.  

Choose a smooth atlas $R\to \mathcal Bun_{(r,d)}$ such that $L_1\oplus L_2$ have preimage in $R$. Thus we have a flat family $\mathcal E$ of vector bundles of rank $2$ and degree
 $d$ parameterized by a smooth curve $T$ (cf. \cite[13.4.A.(b)]{Vakil25} and \cite{Ram78}) such that $\mathcal E|_{t}$ are general stable for $t\in T-\{t_0\}$ and   $\mathcal E|_{t_0}\cong L_1\oplus L_2$.

 Let us consider $\mathrm {fr}^{(N)*}_T \mathcal E$. By shrinking $T$, we may assume $\mathrm {fr}^{(N)*}_T \mathcal E|_{T-\{t_0\}}\cong \mathcal V_1\oplus \mathcal V_2$ with $\deg (\mathcal V_1)=d_1$, $\deg (\mathcal V_2)=p^Nd-d_1$ and $\mathrm {fr}^{(N)*}_T \mathcal E|_{t_0}\cong \mathrm {fr}^{(N)*}L_1\oplus\mathrm {fr}^{(N)*}L_{2}$. By parabolic reduction, (cf. \cite[Proposition 2.2]{Hein08} and \cite[Proposition 7.1.3]{Behrend91}) there is a filtration $\tau$ on $\mathrm {fr}^{(N)*}_T \mathcal E|_{t_0}$ with graded pieces are line bundles of degree $d_1$ and $p^Nd-d_1$. The difference of degrees of these line bundles are larger than $2g-2$, thus the line bundles must split (cf. \cite[Lemma 2.4]{BPRS24}). That is, we have an isomorphism $\mathrm {fr}^{(N)*}L_1\oplus\mathrm {fr}^{(N)*}L_{2}\cong M_1\oplus M_2$ with $\deg(M_1)=d_1$, $\deg(M_2)=p^Nd-d_1$. This isomorphism will lead a contradiction since the decomposition into indecomposable bundles is unique up to isomorphism and permutation as in \cite[LEMMA 2.3]{Schiffmann16} and \cite[Theorem 2, Lemma 6,7]{Atiyah56}. 
\end{proof}

Then we consider the rank $r\geq 3$ cases. We assume   $\mathrm {rank}(\mathrm {gr}_i^{\tau_{N}}\mathrm {fr}^{(N)^*}_UE_U^\text{univ})=r_i$ and  $\mathrm {deg}(\mathrm {gr}_i^{\tau_{N}}\mathrm {fr}^{(N)^*}_UE_U^\text{univ})=d_i$ for $i=1,...,s$. Then $\sum_{i=1}^sr_i=r$, $\sum_{i=1}^s{d_i}=p^{N}d$.

\begin{lemma} \label{lemma:keylem} There exists integers $\ell_1,...,\ell_r\in \mathbb Z$ such that  $\ell_1+\cdots+\ell_r=d$, $\ell_i\geq \ell_{i+1}+2g-2$ and $\ell_{r-1}\geq 0$ such that for any permutation $\sigma\in S_r$, $(\ell_1,...,\ell_r)$ is not a solution to the equation
	\begin{equation}\label{eq:degeq}
		\begin{cases}
			x_{\sigma(1)}+\cdots+x_{\sigma(r_1)}&=d_1/p^{N}\\
			x_{\sigma(r_1+1)}+\cdots+x_{\sigma(r_1+r_2)}&=d_2/p^{N}\\
			\cdots\cdots&\\
			x_{\sigma(r_1+\cdots+r_{s-1}+1)}+\cdots+x_{\sigma(r)}&=d_s/p^{N}
		\end{cases}
	\end{equation}
\end{lemma}

\begin{proof} Let us consider the first $r-1$ variables $x_1,...,x_{r-1}\in \mathbb Q$ and take $x_r=d-x_1-\cdots-x_{r-1}$. If the first $r-1$ variables are in $\mathbb Z$, then $x_r$ is in $\mathbb Z$. Let us consider the equations in \ref{eq:degeq}. For fixed permutation $\sigma$, these equations define a proper affine subspace $W_{\sigma}$ in $\mathbb Q^{r-1}$ (with Zariski topology) and let $\mathrm {ref}_i$ be the reflection of the $i$-th coordinate for $1\leq i\leq r-1$. Then we have a non-empty Zariski open subset
$$U=\mathbb A^{r-1}_{\mathbb Q}-\cup_{i=1}^{r-1}\mathrm {ref}_i[\bigcup_{\sigma\in S_{r}} W_{\sigma}\cup \ \bigcup_{1\leq i<j\leq r-1} V(x_i-x_j)]$$
Since the integral points $\mathbb Z^{r-1}$ are dense in  $\mathbb A^{r-1}_{\mathbb Q}$, there is an integer point $(y_1,...,y_{r-1})$ in $U$. Note that $U$ is invariant under the coordinate reflection and we may assume $y_1>y_2>\cdots>y_{r-1}$. Moreover, $\mathbb A^1_{\mathbb Q}\cdot (y_1,...,y_{r-1})\cap U$ is non-empty thus a Zariski open subset of  $\mathbb A^1_{\mathbb Q}$. For $t$ sufficiently large, we can scale the tuple $(y_1,...,y_{r-1})$ to $(ty_1,...,ty_{r-1})\in U$. Then we find those $(\ell_1,...,\ell_{r})$.  
\end{proof}

As an application, we could show the following result.

\begin{proposition}\label{prop:genricss} 
A general stable bundle $E$ of rank $r$ and degree $d$ in $\mathcal Bun_{r,d}$, $\mathrm {fr}^*E$ is semi-stable. 
\end{proposition}

\begin{proof} 
Consider the unstable bundle $L_1\oplus\cdots\oplus L_r$ in $\mathcal Bun_{r,d}$ with $L_i\in \mathrm {Pic}_{X^{(N_0)}}^{\ell_i}$. Choose a smooth atlas $R\to \mathcal Bun_{(r,d)}$ such that   $L_1\oplus\cdots\oplus L_r$ have a preimage in $R$. Thus we have a flat family $\mathcal E$ of vector bundles of rank $r$ and degree
 $d$ parameterizing by a smooth curve $T$ (cf. \cite[13.4.A.(b)]{Vakil25} and \cite{Ram78}) such that $\mathcal E|_{t}$ are general stable for $t\in T-\{t_0\}$ and   $\mathcal E|_{t_0}\cong L_1\oplus\cdots\oplus L_r$. 
 
 Now we consider $\mathrm {fr}^{(N)*}_T \mathcal E$. By shrinking $T$, we may further assume $\mathrm {fr}^{(N)*}_T \mathcal E|_{T-\{t_0\}}\cong \mathcal V_1\oplus\cdots\oplus \mathcal V_s$ with $\mathrm{rank}(\mathcal V_i) =r_i, \deg (\mathcal V_i)=d_i$ and $\mathrm {fr}^{(N)*}_T \mathcal E|_{t_0}\cong \mathrm {fr}^{(N)*}L_1\oplus\cdots\oplus \mathrm {fr}^{(N)*}L_{r}$. By the parabolic reduction (cf. \cite[Proposition 2.2]{Hein08} and \cite[Proposition 7.1.3]{Behrend91}), there is a filtration $\tau$ on $\mathrm {fr}^{(N)*}_T \mathcal E|_{t_0}$ with graded pieces of the types $(r_i,d_i)$, $i=1,...,s$. Then we have an $\mathbb A^1$ family of bundles whose generic fibers are isomorphic to $\mathrm {fr}^{(N)*}_T \mathcal E|_{t_0}$ and its central fiber isomorphic to $\oplus_{i=1}^s\mathrm {gr}_i^\tau$. On the other hand, the splitting of line bundles induces a Borel reduction thus there is an other filtration on $\oplus_{i=1}^s\mathrm {gr}_i^\tau$ with graded pieces are line bundles of degree $p^N\ell_i$, $i=1,...,r$. The difference of degrees of these line bundles are larger than $2g-2$, thus the line bundles must split (cf. \cite[Lemma 2.4]{BPRS24}). That is we have an isomorphism $\oplus_{i=1}^s\mathrm {gr}_i^\tau\cong \oplus_{i=1}^r  M_{i}$ with $M_i\in \mathrm {Pic}_X^{p^N\ell_i}$. Then by the uniqueness of indecomposable decomposition as in \cite[LEMMA 2.3]{Schiffmann16} and \cite[Theorem 2, Lemma 6,7]{Atiyah56}, the isomorphism $\oplus_{i=1}^s\mathrm {gr}_i^\tau\cong \oplus_{i=1}^r  M_{i}$ shows that $(\ell_1,...,\ell_r)$ is a solution to the equation (\ref{eq:degeq}) for some permutation $\sigma\in S_r$, thus we get a contradiction by Lemma \ref{lemma:keylem}. 
 \end{proof}

Stability is invariant under tensoring by a line bundle; therefore we have:

\begin{corollary} \label{cor:genpullbackss} 
Given a general stable bundle $E$ of rank $r$ and $\det E\cong \mathcal O_{X^{1}}(dx_0)$ in $\mathcal Bun_{\mathrm {SL}_r,\mathcal O_{X^{(1)}}(dx_0)}$, then $\mathrm {fr}^*E$ is semi-stable. 
\end{corollary}

and

\begin{corollary}\label{cor:rationalmap} 	Let $X$ be a smooth projective curve of genus $g\geq 2$ over an algebraically closed field $k$ of characteristic $p >0$, and $\mathrm {fr} : X \to X^{(1)}$ the relative Frobenius morphism. Then the Frobenius pullback define rational maps $$\phi_{\mathrm{GL}_r,d}:\mathrm {Bun}_{X^{(1)},(r,d)}^\text{ss} \dashrightarrow \mathrm {Bun}^\text{ss}_{X,(r,pd)}\ \text{and}\ \phi_{\mathrm{SL}_r,d}:\mathrm {Bun}_{\mathrm {SL}_r,d}^\text{ss} \dashrightarrow \mathrm {Bun}^\text{ss}_{\mathrm {SL}_r,pd},\ E\mapsto \mathrm {fr}^*E.$$
\end{corollary}

\subsection{Proof of main theorem}

We are now ready to prove the main result.

\begin{theorem}\label{Thm:main1}
Let $X$ be a smooth projective curve of genus $g\geq 2$ over an algebraically closed field $k$ of characteristic $p >0$, and $\mathrm {fr} : X \to X^{(1)}$ the relative Frobenius morphism. 
\begin{itemize}
\item[(a)] the Frobenius pullback $\mathrm {fr}^*E$ of a general semi-stable vector bundle $E$ with rank $r$ and degree $d$ in $\mathrm {Bun}^\text{ss}_{(r,d)}$ is still semi-stable.  
\item[(b)] the set-theoretic map \[\phi^\mathrm {Bun}_{(r,d)}:\mathrm {Bun}_{X^{(1)},(r,d)}^\text{ss} \dashrightarrow \mathrm {Bun}^\text{ss}_{X,(r,pd)},\ E\mapsto \mathrm {fr}^*E\] induced by the Frobenius pullback is a dominant rational map on the moduli spaces.
\end{itemize}
\end{theorem}
\begin{remark} In general, the rational map $\phi^\mathrm {Bun}_{(r,d)}$ is not defined over $\mathrm {Bun}^\text{ss}_{(r,d)}$. If $d=0$,  Gieseker \cite{Gieseker73} showed that there is a vector bundle $E\in \mathrm {Bun}_{(r,0)}^\text{ss}$ such that $\mathrm {fr}^*E$ is unstable. For $d\neq 0$, a Frobenius direct image of a line bundle $\mathrm {fr}_*L$ is semistable but its Frobenius pullback $\mathrm {fr}^*\mathrm {fr}_*L$ is always unstable (cf. \cite{Sun08}).  
\end{remark}

By Corollary \ref{cor:rationalmap}, for the proof of Theorem \ref{Thm:main1}, we have to show that $\phi$ is dominant.

\begin{lemma}\label{lemma:stabledom} Let $X$ be a smooth projective curve of genus $g\geq 2$ over an algebraically closed field $k$ of characteristic $p >0$, if there is a stable bundle $V$ of rank $r$ and determinant $\mathcal O_X(pd x_0)$ such that the canonical connections on $V$ is parameterized by a non-empty scheme, then the rational maps $$\phi_{\mathrm {SL}_r,d}:\mathrm {Bun}_{\mathrm{SL}_r,d}^\text{ss} \dashrightarrow \mathrm {Bun}^\text{ss}_{\mathrm {SL}_r,pd}\ \text{and}\  \phi_{\mathrm {GL}_r,d}:\mathrm {Bun}_{X^{(1)},(r,d)}^\text{ss} \dashrightarrow \mathrm {Bun}^\text{ss}_{X,(r,pd)}$$ are dominant.
\end{lemma}

\begin{proof}  We can first consider the $\mathrm {SL}_r$ bundle case. Here we consider $\mathrm {SL}_r$ bundle because the stabilizer of a stable bundle in the moduli stack is finite. In this case, the map $\{V\}\to \mathcal Bun_{\mathrm {SL}_r,d}$ is universally closed. 

Let $E$ be a stable bundle such that $V=\mathrm {fr}^*E$ is stable and we denote the canonical connection by $\nabla^\text{can}$. Consider the forgetful map
	$$\mathtt {fg}: \mathcal M_{\mathrm{dR},\mathrm {SL}_r,pd}^\text{ss}\to \mathcal Bun_{\mathrm {SL}_r,pd},\ (F,\nabla)\mapsto F.$$
	 For a stable bundle $V$ the fiber  $\mathtt{fg}^{-1}V$ at the closed point $V\in \mathcal Bun^\text{s}_{\mathrm {SL}_r,pd}$ is an affine space $\Gamma=\Gamma(X,\mathrm {End}^0(V)\otimes \Omega_X^1)$ and the map $\mathtt{fg}^{-1}\{V\}\to \mathcal M^\text{ss}_{\mathrm {dR},\mathrm {SL}_r,d}$ is universally closed. 
	
	Then consider the fiber product
	$$\begin{tikzcd}
		C\arrow[r] \arrow[d] \arrow[dr, phantom, "\ulcorner"]
		& \mathrm F^ \mathcal B(\mathcal Bun^\text{ss}_{\mathrm {SL}_r,d}) \arrow[d] \\
		\Gamma \cdot[V,\nabla^\text{can} ]=\mathtt {fg}^{-1}(V)\arrow[r]
		& \mathcal M_{\mathrm{dR},\mathrm {SL}_r,pd}^\text{ss}
	\end{tikzcd}$$
	Here $C$ is the space of canonical connections on $V$. $C$ is non-empty since there is a canonical connection on $V$. By universally closed property of $\mathrm F^ \mathcal B(\mathcal Bun^\text{ss}_{\mathrm {SL}_r,d})$, $C$ is a (universally closed, separated and finite type thus) projective subscheme in the affine scheme $\Gamma$ thus finite. By this, we see that the canonical connections on $V$ are parameterized by a finite scheme over $k$. So by Cartier descent, the fiber $\phi^{-1}_{\mathrm {SL}_r,d}(V)$ at $V$  of the scheme map $\phi_{\mathrm {SL}_r,d}$ is isomorphic to the scheme parameterizing canonical connections on $V$. Thus the fiber dimension of $\phi$ is zero and $\phi$ is dominant.  
\end{proof}

We begin with the case where the rank and degree are coprime.

\begin{lemma}\label{lem:1rdcoprime} Let rank $r$ and degree $d$ are coprime, and $r$ and $p$ are coprime. The Frobenius pull back induced map $\phi_{\mathrm {SL}_r}:\mathrm {Bun}_{\mathrm {SL}_r,d}^\text{ss} \dashrightarrow \mathrm {Bun}^\text{ss}_{\mathrm {SL}_r,pd}$ is dominant.
\end{lemma}

\begin{proof} In the case that $(r,d)=1$ and $(r,p)=1$, for both of the domain and target of $\phi$, the stability conditions coincide with the semi-stability conditions and both moduli spaces are smooth and projective. Thus by Corollary \ref{cor:genpullbackss}, there is a stable bundle posses a canonical connetion. So by Lemma \ref{lemma:stabledom}, we have the proof. \end{proof}
	
%

		\begin{lemma}\label{lem:dominant}  For given two positive integers $p$ and $d$, which are not necessarily coprime, the Frobenius pull back map $\phi_{\mathrm {SL}_r}:\mathrm {Bun}_{\mathrm {SL}_r,d}^\text{ss} \dashrightarrow \mathrm {Bun}^\text{ss}_{\mathrm {SL}_r,pd}$ is dominant.
	\end{lemma}

	\begin{proof} In this case, we should again consider the defomation technique in Proposition \ref{prop:genricss}. Let us assume the generic Frobenius pullback bundle $\mathrm {fr}^{*}E$ is strictly semi-stable, that is it has non-trivial Jordan-H\"oder filtration. 
    
    Let $\eta\in \mathcal Bun_{\mathrm {SL}_r,d}$ be the generic point. We may assume $\mathrm {fr}^*_\eta E^\text{univ}_\eta$ must be strictly semi-stable. If it is stable, then there is an open subset $U$ in  $\mathcal Bun_{\mathrm {SL}_r,d}$, such that for all $E\in U$, $\mathrm {fr}^*E$ is stable and by Lemma \ref{lemma:stabledom} we get the proof. Since $\mathrm {fr}^*_\eta E^\text{univ}_\eta$ is strictly semi-stable, we have a non-trivial subbundle and quotient bundle of $\mathrm {fr}^*_\eta E^\text{univ}_\eta$ with both of them semi-stable with slope $\frac{pd}{r}$. The filtration spread out to an open subset $U$, thus the bundle $\mathrm {fr}^*_U E^\text{univ}_U$ have a non-trivial filtration of vector bundles such that the graded pieces are semi-stable with slope $\frac{pd}{r}$.

Thus we can pick an unstable bundle $L_1\oplus\cdots\oplus L_r$ in $\mathcal Bun_{r,d}$ with certain bad degree. Choose a smooth atlas $R\to \mathcal Bun_{(r,d)}$ such that   $L_1\oplus\cdots\oplus L_r$ have a preimage in $R$. Thus we have a flat family $\mathcal E$ of vector bundles of rank $r$ and degree
 $d$ parameterizing by a smooth curve $T$ such that $\mathcal E|_{t}$ are general stable for $t\in T-\{t_0\}$ and   $\mathcal E|_{t_0}\cong L_1\oplus\cdots\oplus L_r$. Now we consider $\mathrm {fr}^{*}_T \mathcal E$. By shrinking $T$, we may further assume $\mathrm {fr}^{*}_T \mathcal E|_{T-\{t_0\}}$ have a non-trivial filtration of vector bundles such that the graded pieces are semi-stable with slope $\frac{pd}{r}$. By the parabolic reduction (cf. \cite[Proposition 2.2]{Hein08} and \cite[Proposition 7.1.3]{Behrend91}), there is a filtration $\tau$ on $\mathrm {fr}^{*}_T \mathcal E|_{t_0}$ with graded pieces of slope $\frac{pd}{r}$. Then we can construct an $\mathbb A^1$ family of bundles whose generic fibers are isomorphic to $\mathrm {fr}^{*}_T \mathcal E|_{t_0}$ and its central fiber isomorphic to $\oplus_{i=1}^2\mathrm {gr}_i^\tau$. On the other hand, the splitting of line bundles induces a Borel reduction thus there is an other filtration on $\oplus_{i=1}^2\mathrm {gr}_i^\tau$ with graded pieces are line bundles of degree $p\deg L_i$, $i=1,...,r$. We may assume the difference of degrees of these line bundles are larger than $2g-2$, thus the line bundles must split (cf. \cite[Lemma 2.4]{BPRS24}). Then we have an isomorphism $\oplus_{i=1}^2\mathrm {gr}_i^\tau\cong \oplus_{i=1}^r  M_{i}$ with $\deg M_i=p\deg L_i$. Then by the uniqueness of indecomposable decomposition as in \cite[LEMMA 2.3]{Schiffmann16} and \cite[Theorem 2, Lemma 6,7]{Atiyah56}, the isomorphism $\oplus_{i=1}^s\mathrm {gr}_i^\tau\cong \oplus_{i=1}^r  M_{i}$ shows that $(\ell_1,...,\ell_r)$ is a solution to the equation 
 $$\begin{cases}
     x_1+\cdots+x_{r_1}=&\frac{pdr_1}{r}\\
     x_{r_1+1}+\cdots+x_{r_2}=&\frac{pdr_2}{r}
 \end{cases}.$$  Thus by a good choice of the degree of $L_i$'s as in Lemma \ref{lemma:keylem}, we get a contradiction. Thus there is an open subset $U$ in  $\mathcal Bun_{\mathrm {SL}_r,d}$, such that for all $E\in U$, $\mathrm {fr}^*E$ is stable and by Lemma \ref{lemma:stabledom} we get the proof. 
\end{proof}

\section{Corollaries of the main theorem}

In this section, we give some corollaries of the Theorem \ref{Thm:main1}.

\subsection{Semistability of Frobenius direct image sheaves} Recall the following Faltings - Le Potier's cohomological criterion of the semistability. 
\begin{proposition}[cf. \cite{Fal93} and Th\'eor\'eme 2.4, 2.5 in \cite{LeP96}]\label{prop:FL} Let $E$ be a vector bundle of rank $r$ and degree $d$ over $X$, let $r_1=\frac{r}{\mathrm {gcd}(r,d)}$. Then
\begin{itemize}
\item[1.] If there exists a vector bundle $V$ with $\mu(V)+\mu(E)=(g-1)$ such that $h^0(X,E\otimes V)=h^1(X,E\otimes V)=0$, then $E$ and $V$ are both semistable.
\item[2.] If $E$ is semistable, then for any integer $\ell>\frac{r^2}{4}(g-1)$, a general vector bundle $V$ in the moduli space of semistable bundles with rank $\frac{\ell r}{\mathrm {gcd}(r,d)}=\ell\cdot r_1$ and degree $\ell r_1(g-1-\mu(E))$(this degree condition is equivalent to $\mu(V)=g-1-\mu(E)$) has the property that $h^0(X,E\otimes V)=h^1(X,E\otimes V)=0$.
\end{itemize}
\end{proposition}
Now we can prove the following corollary (cf. \cite[Theorem 1.1]{MP07}\cite[Theorem 2.2]{Sun08}) by Theorem \ref{Thm:main1} and without using the covering trick or inequalities.

\begin{corollary} Assume $X$ is a smooth projective curve over an algebraic closed field $k$ with $\mathrm{char}(k)=p >0$. Then the Frobenius direct image of a semistable bundle over $X$ is semistable.
\end{corollary}
\begin{proof} From Theorem \ref{Thm:main1}, we know that the Frobenius pull back map on the moduli spaces is rational and dominant. Hence a general vector bundle $$G \in \mathrm{Bun}^\text{ss}_{(\ell r_1,p\ell r_1\cdot (g-1-\mu(E)))}$$ is of the form $\mathrm {fr}^* G^{\prime}$ for some $$G^{\prime} \in \mathrm{Bun}^\text{ss}_{(X^{1},\ell r_1,\ell r_1\cdot (g-1-\mu(E)))}.$$ Now let $E$ be a semistable bundle with rank $r$ and degree $d$. Then by Proposition \ref{prop:FL}, $h^0(X, E \otimes G)=h^1(X,E\otimes G)=0$ for a general $G \in \mathrm{Bun}^\text{ss}_{(\ell r_1,p\ell r_1\cdot (g-1-\mu(E)))}$. Assuming that $G$ is general, we can write $G=\mathrm {fr}^* G^{\prime}$ and by adjunction we obtain
$$
h^i\left(X, E \otimes \mathrm {fr}^* G^{\prime}\right)=h^i\left(X^{(1)}, \mathrm {fr}_* E \otimes G^{\prime}\right)=0,\quad i=0,1.
$$
This shows that $\mathrm {fr}_*E$ is semistable by Proposition \ref{prop:FL}.
\end{proof}

\subsection{Strongly semistable bundle is very general} In the case of the positive characteristic, we have a very important notion of the stability condition, namely the strongly semistability. Given an ample coherent sheaf $H$, a sheaf $E$ is called strongly slope $H$-semistable, if for any positive integer $\ell$, the pull back $(\mathrm {fr}^*)^{\ell}E$ is slope $H$-semistable. For basic properties of strongly semi-stable sheaves, we refer the readers to \cite{Lan08,Lan09} and references there in. 

 By our Theorem \ref{Thm:main1}, we have the following corollary.
\begin{corollary} Assume $X$ is a smooth projective curve over an algebraic closed field $k$ with $\mathrm{char}(k)=p>0$. If the field $k$ has uncountably many elements, then the strongly semistable bundle is non empty in the moduli space of semistable bundles.
\end{corollary}
\begin{proof} Let $U_\ell=(\mathrm {F}^{\mathrm{Bun},\circ \ell})^{-1}(\mathrm {Bun}^\text{ss}_{(r,p^\ell d)} )$, then by Theorem \ref{Thm:main1} we see that $U_\ell$ is the open subset in $\mathrm {Bun}_{(r,d)}^\text{ss}$ parametrize bundle $E$ such that $(\mathrm {fr}^*)^{\ell}E$ is semistable. Then the set of strongly semistable bundles equals to $\cap_{\ell=1}^\infty U_\ell$. Thus by the Hint in \cite[Chapter V, Section 4, Exercise 4.15 (c)]{Hart77} the intersection of countable open subsets is non-empty.  
\end{proof}

\subsection{Moduli space of {$\lambda$}-connections}

A Higgs bundle on a curve $X$ is a pair $(E,\theta)$, where $E$ is a vector bundle and $\theta:E\to E\otimes \Omega_X^1$ is an $\mathcal O_X$-linear map, which is called a Higgs field. In general, a Higgs field can take values in a coherent sheaf $F$, i.e. $\theta$ can be an $\mathcal O_X$-linear map $E\to E\otimes_{\mathcal O_X}F$.
	In the case that $F=\Omega_X^1$, by Serre duality, a Higgs field $$\theta\in\Gamma(X,\mathcal End_{\mathcal O_X}(E)\otimes \Omega_X^1)\cong \mathrm H^1(X,\mathcal End_{\mathcal O_X}(E))^\vee$$ can be regarded as a cotangent vector in $T^\vee_{\mathcal  N}|_{[E]}$ and the moduli of Higgs bundles can be regarded as the cotangent space of the moduli of vector bundles (cf. \cite{BD}). 
	
	The notion of Higgs bundle was introduced by Hitchin in \cite{Hit87} as the solution to the self-dual Yang-Mills equations. The other ways come from Simpson and Deligne in \cite{Sim87}, in their minds, the concept of a Higgs bundle comes from taking gradings of a variation of Hodge structures. Then Simpson use Higgs bundles to built up a non-abelian verison of Hodge theory in \cite{Sim92}.  The correspondence induces a $C^\infty$-homeomorphism of the moduli spaces of Higgs bundles, flat connections and $\pi_1$-representations. 
	
	N. Nitsure in \cite{N91} constructed the moduli space of (semi-)stable Higgs bundles and we denote our $\Omega_X^1$ valued Higgs bundle moduli by $M_{\text{Dol}, (r,d)}^\text{ss}$ and call it as the Dolbeault moduli after C. Simpson in \cite{Sim92}. If the rank and degree are coprime,  $M_{\text{Dol}, (r,d)}^\text{ss}$ is smooth and quasi-projective. In characteristic zero case, even $r,d$ are not coprime, by Simpson correspondence (cf. \cite[Section 11]{Sim94II}), one can deduce the normality and irreducibility of $M_{\text{Dol}, (r,d)}^\text{ss}$ by passing to the moduli space of $\pi_1$ representations. In the positive characteristic case, we have the following result.
	\begin{proposition}\label{prop:conndol}If the rank $r$ and degree $d$ are not coprime, we assume that $r=r_1q$ and $d=d_1q$ such that $(r_1,d_1)=1$, then the moduli space  $M_{\text {Dol},(r,d)}^\text{ss}$ of $\Omega_X^1$-valued Higgs bundles on a curve $X$ with genus $\geq 2$ is connected. Moreover, the fibers of the Hitchin map $\mathtt h_{\text{Dol},(r,d)}:M_{\text{Dol},(r,d)}^\text{ss}\to A(X)$ with $A(X)=\prod_{i=1}^r\Gamma(X,(\Omega_X^1)^{\otimes i})$ are connected. 
	\end{proposition}    
	
	\begin{proof} We first point out that the moduli space of stable Higgs bundles $M_{\text{Dol}, (r,d)}^\text{s}\subset M_{\text{Dol}, (r,d)}^\text{ss}$ is smooth by deformation theory (cf. \cite[Page 3, end of the 2nd paragraph]{Hein15}). We check that   $M_{\text{Dol}, (r,d)}^\text{s}$ is connected by dimension estimate. Let $U_\text{int}\subset A(X)$ be the $\mathrm {codim}\geq 2$ open subset  parameterize  points in the Hitchin base with integral spectral curves and $Z_\text{int}$ be its complement. Then $\mathtt{h}_{\text{Dol}}^{-1}(U_\text{int})$ is irreducible because the fibers are compactified Jacobians of integral curves thus irreducible and the base $U_\text{int}$ is irreducible. Thus the connected components do not intersect   $\mathtt{h}_{\text{Dol}}^{-1}(U_\text{int})$ are contained in $\mathtt{h}^{-1}_{\text{Dol}}(Z_\text{int})$ and have dimension equals to $\dim(M_{\text{Dol}},(r,d))$ if non-empty by the smoothness. But $\mathtt{h}^{-1}_{\text{Dol}}(Z_\text{int})$ is of dimension $\leq \mathrm {dim}(M_{\text{Dol},(r,d)})-2$ because the fibers of the Hitchin map is of constant dimension $\frac{1}{2} \mathrm {dim}(M_{\text{Dol},(r,d)})$ (cf. \cite{Lau88}). So $M_{\text{Dol},(r,d)}^\text{s}$ must be connected.    
		
		 Let us show the connectedness of $M_{\text{Dol}, (r,d)}^\text{ss}$ by induction on the rank $r$. If $r=1$, then $M_{\text{Dol},(1,d)}^\text{ss}\cong \mathrm {Pic}^d_X\times \Gamma(X,\Omega_X^1)$, so $M_{\text{Dol},(1,d)}^\text{ss}$ is connected in this case. Let us assume the connectedness of $M_{\text{Dol}, (r,d)}^\text{ss}$ for $r<r_0$. Consider the connectedness of $M_{\text{Dol}, (r_0,d_0)}^\text{ss}$. Recall that $\mathrm {Bun}^\text{ss}_{(r_0,d_0)}$ is irreducible and normal (cf.\cite[PREMI\`ERE PARTIE, TH\'EOR\`EME 17 ]{Seshadri82}) and it is embedded as a closed subvariety of $M_{\text{Dol},(r_0,d_0)}^\text{ss}$ parameterize those semistable Higgs bundles with zero Higgs fields. Since $\mathrm {Bun}^\text{ss}_{(r_0,d_0)}$ is connected, we just have to check that every semistable Higgs bundle can be deformed to a semistable bundle in $\mathrm {Bun}^\text{ss}_{(r_0,d_0)}$. Let $(E,\theta)$ be a Higgs bundle, if it is stable, then by the connectedness of $M_{\text{Dol},(r_0,d_0)}^\text{s}$, it can deform to a semistable Higgs bundle in  $\mathrm {Bun}^\text{ss}_{(r_0,d_0)}$. If $(E,\theta)$ is strictly semistable, we may assume it is polystable, that is $(E,\theta)\cong \oplus_{i=1}^\ell(E_i,\theta_i)$ with $(E_i,\theta_i)$ are all stable Higgs bundle with slope $\frac{d_0}{r_0}$. Let $r_i=\mathrm {rank}(E_i)$ and $d_i=\mathrm {deg}(E_i)$, since $r_i$'s are strictly less than $r_0$, so by induction, we can deform $(E_i,\theta_i)$ to a semistable vector bundle $E'_i$ with slope $\frac{d_0}{r_0}$. Then take direct sum and we get that $\oplus_{i=1}^\ell(E_i,\theta_i)$ can be deform to a semistable vector bundle $\oplus_{i=1}^\ell(E_i',0)$, so  $M_{\text{Dol},(r,d)}^\text{ss}$ is connected.
		 
		 Then by hyperbolic localization technique in \cite{FHZ24}, $M_{\text{Dol},(r,d)}^\text{ss}$ and its global nilpotent cone have the same cohomology, that the global nilpotent cone of $M_{\text{Dol},(r,d)}^\text{ss}$ is connected. Then by \cite[tag 055H]{stacks-project}, all fibers of $\mathtt{h}_{\text{Dol},(r,d)}$ are connected. This is because for any fiber $\mathtt{h}^{-1}(a)$, there is a curve, which is the orbit $\mathbb A^1\cdot a$. If we restrict the Hitchin map to the curve, by the $\mathbb G_m$ action on $M_{\text{Dol},(r,d)}^\text{ss}$, the zero fiber is the global nilpotent cone and the other fibers are isomorphic to  $\mathtt{h}^{-1}(a)$. 
	\end{proof}

	Deligne and Simpson in \cite{Sim98,Sim10} define the concept of $\lambda$-connections realising the Higgs bundles as a degeneration of vector bundles with flat connections. Here, for $\lambda\in \mathbb C$, a $\lambda$-connection on a vector bundle $E$ consists of an operator $D_\lambda:E\to E\otimes \Omega_X^1$ such that $D_\lambda(fe) = \lambda e\otimes \mathrm df  + fD_\lambda(e)$ (Leibniz rule multiplied by $\lambda$) and such that $D_\lambda^2 = 0$(integrability) as defined in the usual way. Note that if $\lambda = 1$ then this is the same as the usual notion of a flat connection, whereas if $\lambda = 0$ then this is the same as the notion of Higgs field making $(E, D_0)$	into a Higgs bundle. Moreover, Simpson in \cite{Sim98} further introduced the moduli space of all $\lambda$-connections for $\lambda\in \mathbb A^1_{\mathbb C}$ which he called Hodge moduli space and denote it by $M_\text{Hdg}$ and regard the map $M_\text{Hdg}\to \mathbb A^1,\ {D_\lambda}\mapsto \lambda$ as the Hodge filtration on $M_\text{dR}$. 
	
	In positive characteristic, the moduli of $\lambda$-connections are also studied by \cite{LP01,Lan14,CZ15,Groechenig16,Lan22,dCZ22A,dCZ22P,FZ23,dCGZ24,dCFHZ24}. We refer the readers to these papers and the references their in. We point out here that unlike the characteristic zero case, a vector bundle with non-zero degree may admit a flat connection. According to \cite{Atiyah57} (see also \cite{BS06}), a vector bundle $E$ admit a connection if and only if its Atiyah class $\mathrm {At}(E)\in \mathrm {H}^1(X,\mathcal End(E))$ vanish. This implies that in characteristic $p>0$ case, if $E$ admit a connection, then its degree must be divided by $p$. In this case, we have shown that.
	
	\begin{corollary} If the rank $r$ and degree $d$ are not coprime, we assume that $r=r_1q$ and $d=d_1q$ such that $(r_1,d_1)=1$, then the moduli space  $M_{\text {dR},(r,d)}^\text{ss}$ of flat connections on a curve $X$ with genus $\geq 2$ is connected. Moreover, the fibers of the Hodge-Hitchin map (cf. \cite{LP01,dCZ22P,Lan22}) $\mathtt h_{\text{Hdg},(r,pd)}:M_{\text{Hdg},(r,pd)}^\text{ss}\to A(X^{(1)})\times \mathbb A^1$ are connected. 
	\end{corollary}
	\begin{proof} By the very good splitting theorem \cite[Corollary 4.14]{dCGZ24}, the $\mathbb G_m$ action on $\mathtt{h}_{\text{Hdg}}$ together with Proposition \ref{prop:conndol}, we could get the desired result.   
	\end{proof}
	If we restrict the Hodge-Hitchin map to the loci of nilpotent $p$-curvatures, we have
	$\mathtt h_{\text{Nilp,Hdg}}:\mathrm{Nilp}_{\text{Hdg},(r,pd)}^\text{ss}\to\mathbb A^1$ and $\mathtt h_{\text{Nilp,Hdg}}^{\psi \leq \ell}:\mathrm{Nilp}_{\text{Hdg},(r,pd)}^{\text{ss},\psi\leq \ell}\to\mathbb A^1$ the loci of $p$-curvature nilpotence exponent $\leq \ell$. In particular, if $\ell =1$, as pointed out by Langer in \cite[Page 531]{Lan14}, $(\mathtt{h}_{\text{Nilp,Hdg}}^{\psi\leq 1})^{-1}(0)\cong [\mathtt{h}_{\text{Dol}}^{-1}(0)]^{\theta\leq p}$. That is, $(\mathtt{h}_{\text{Nilp,Hdg}}^{\psi\leq 1})^{-1}(0)$ is not the loci with $\theta=0$, but the loci with $\theta^p=0$. By our main theorem, we have.
	\begin{corollary} There is an irreducible closed subset $N$ in $\mathrm{Nilp}_{\text{Hdg},(r,pd)}^{\text{ss},\psi\leq 1}$, such that $\mathtt{h}_\text{Nilp,Hdg} (N)=\mathbb A^1$ and $\mathrm {Bun}^\text{ss}_{(r,d)} \cong[\mathtt{h}_{\text{Dol}}^{-1}(0)]^{\theta=0}\subset N$. 
	\end{corollary}	 
\begin{proof} By Theorem \ref{Thm:main1} and \cite[Proposition 3.5]{Lau88}, we can pick an open subset $U_\text{vst,f}\subset \mathrm {Bun}^\text{s}_{(r,pd)}$ such that for any $[F]\in U_\text{vst,f}$, $F$ is very stable and there is a stable bundle $E$ such that $F\cong \mathrm {fr}^*E$. Thus on $X\times U_\text{vst,f}\times \mathbb A^1$, we have a family of $\lambda$-connections defined by $(F,\lambda\nabla^\text{can}_F)$. This $\lambda$ connection defines an immersion $U_\text{vst,f}\times \mathbb A^1\to  \mathrm{Nilp}_{\text{Hdg},(r,pd)}^{\text{ss},\psi\leq 1}$ over $\mathbb A^1$. Take $N$ to be the closure of $U_\text{vst,f}\times \mathbb A^1$. 
\end{proof}

\subsection{Frobenius stratification} The Harder-Narasimhan filtration type of Frobenius pull backs of semistable vector bundles defines a stratification on $\mathrm {Bun}^\text{ss}_{(r,d)}$, which is called the Frobenius stratification. For the basic properties of the Frobenius stratification, we refer the readers to \cite{LP02,JRXY06,LangePauly08,Ducrohet09,Li14,Li19B,Li19T,Li20,LiZhang24} and the references therein. By Theorem \ref{Thm:main1}, we have:
\begin{corollary} The open strata of the minimal polygon, that is $U\subset \mathrm {Bun}_{X,(r,d)}^\text{ss}$ parameterize semistable bundle $E$ such that $\mathrm {fr}^*E$ is semistable is a non-empty subset of $\mathrm {Bun}_{X,(r,d)}^\text{ss}$. 
\end{corollary}

\vspace{5pt}
\noindent \textbf{Acknowledgments}. The author Jin Cao was supported by the University of Science and Technology Beijing Foundation, China (Grant No. 00007886) and the Fundamental Research Funds for the Central Universities (Grant No. 06320202 and Grant No. FRF-BRB-25-007). The author Xiaoyu Su was supported by National Natural Science Foundation of China (Grant No. 12301056) and the Fundamental Research Funds for the Central Universities (Grant No. 510224016).

The authors would like to thank Prof. Yi Gu for very useful discussions. The author Xiaoyu Su would like to thank Prof. Jianyong Qiao and Prof. Yuming Zhong for their generous help.
\vspace{5pt}


\begin{thebibliography}{dCFHZ24}
	
	\bibitem[Ara04]{Ara04}
	Donu Arapura.
	\newblock Frobenius amplitude and strong vanishing theorems for vector bundles
	{With} an appendix by {Dennis} {S}. {Keeler}.
	\newblock {\em Duke Math. J.}, 121(2):231--267, 2004.
	
	\bibitem[Ati56]{Atiyah56}
	Michael~F. Atiyah.
	\newblock On the {Krull}-{Schmidt} theorem with application to sheaves.
	\newblock {\em Bull. Soc. Math. Fr.}, 84:307--317, 1956.
	
	\bibitem[Ati57]{Atiyah57}
	M.~F. Atiyah.
	\newblock Complex analytic connections in fibre bundles.
	\newblock {\em Trans. Amer. Math. Soc.}, 85:181--207, 1957.
	
	\bibitem[BD]{BD}
	Alexander~A. Beilinson and Vladimir~G. Drinfeld.
	\newblock Quantization of {Hitchin}'s integrable system and {Hecke}
	eigensheaves.
	\newblock It is available at
	\url{http://www.math.uchicago.edu/~mitya/langlands/hitchin/BD-hitchin.pdf}.
	
	\bibitem[Beh91]{Behrend91}
	Kai~Achim Behrend.
	\newblock {\em The Lefschetz trace formula for the moduli stack of principal
		bundles}.
	\newblock University of California, Berkeley, 1991.
	
	\bibitem[BMRR08]{BMR08}
	Roman Bezrukavnikov, Ivan Mirkovi\'c, Dmitriy Rumynin, and Simon Riche.
	\newblock Localization of modules for a semisimple lie algebra in prime
	characteristic (with an appendix by {R. Bezrukavnikov} and {S.
		Riche}:computation for {$sl(3)$}).
	\newblock {\em Ann. of Math.}, 167(3):945--991, 2008.
	
	\bibitem[BP04]{BP04}
	Indranil Biswas and A.~J. Parameswaran.
	\newblock On the ample vector bundles over curves in positive characteristic.
	\newblock {\em C. R., Math., Acad. Sci. Paris}, 339(5):355--358, 2004.
	
	\bibitem[BPRS24]{BPRS24}
	L.~Brambila-Paz and R.~R{\'{\i}}os-Sierra.
	\newblock Moduli of unstable bundles of {HN}-length two with fixed algebra of
	endomorphisms.
	\newblock In {\em Moduli spaces and vector bundles -- new trends. VBAC 2022
		conference in honor of Peter Newstead's 80th birthday, University of Warwick,
		Coventry, United kingdom, July 25--29, 2022}, pages 79--102. Providence, RI:
	American Mathematical Society (AMS), 2024.
	
	\bibitem[BS06]{BS06}
	Indranil Biswas and S.~Subramanian.
	\newblock Vector bundles on curves admitting a connection.
	\newblock {\em Q. J. Math.}, 57(2):143--150, 2006.
	
	\bibitem[CZ15]{CZ15}
	Tsao~Hsien Chen and Xinwen Zhu.
	\newblock Non-abelian hodge theory for algebraic curves in characteristic
	{$p$}.
	\newblock {\em Geometric and Functional Analysis}, 25(6):1706--1733, 2015.
	
	\bibitem[dCFHZ24]{dCFHZ24}
	Mark~Andrea de~Cataldo, Andres Fernandez~Herrero, and Siqing Zhang.
	\newblock Geometry of the logarithmic {Hodge} moduli space.
	\newblock {\em J. Lond. Math. Soc., II. Ser.}, 109(1):38, 2024.
	\newblock Id/No e12857.
	
	\bibitem[dCGZ24]{dCGZ24}
	Mark Andrea~A. de~Cataldo, Michael Groechenig, and Siqing Zhang.
	\newblock The de {Rham} stack and the variety of very good splittings of a
	curve.
	\newblock {\em Math. Res. Lett.}, 31(5):1353--1389, 2024.
	
	\bibitem[dCZ22a]{dCZ22A}
	Mark~Andrea de~Cataldo and Siqing Zhang.
	\newblock A cohomological nonabelian {Hodge} theorem in positive
	characteristic.
	\newblock {\em Algebr. Geom.}, 9(5):606--632, 2022.
	
	\bibitem[dCZ22b]{dCZ22P}
	Mark Andrea~A. de~Cataldo and Siqing Zhang.
	\newblock Projective completion of moduli of {{\(t\)}}-connections on curves in
	positive and mixed characteristic.
	\newblock {\em Adv. Math.}, 401:43, 2022.
	\newblock Id/No 108329.
	
	\bibitem[Duc09]{Ducrohet09}
	Laurent Ducrohet.
	\newblock The {Frobenius} action on rank 2 vector bundles over curves in small
	genus and small characteristic.
	\newblock {\em Ann. Inst. Fourier}, 59(4):1641--1669, 2009.
	
	\bibitem[Fal93]{Fal93}
	Gerd Faltings.
	\newblock Stable {$G$}-bundles and projective connections.
	\newblock {\em J. Algebraic Geom.}, 2(3):507--568, 1993.
	
	\bibitem[FHZ23]{FZ23}
	Andres Fernandez~Herrero and Siqing Zhang.
	\newblock Meromorphic {Hodge} moduli spaces for reductive groups in arbitrary
	characteristic.
	\newblock Preprint, {arXiv}:2307.16755 [math.{AG}] (2023), 2023.
	
	\bibitem[FHZ24]{FHZ24}
	Andres Fernandez~Herrero and Siqing Zhang.
	\newblock Topology of {$\mathbb{G}_m$}-actions and applications to the moduli
	of {Higgs} bundles.
	\newblock Preprint, {arXiv}:2408.03275 [math.{AG}] (2024),v2, 2024.
	\newblock Accepted version. To appear in Math. Res. Lett.
	
	\bibitem[Gie71]{Gieseker71}
	David Gieseker.
	\newblock p-ample bundles and their {Chern} classes.
	\newblock {\em Nagoya Math. J.}, 43:91--116, 1971.
	
	\bibitem[Gie73]{Gieseker73}
	David Gieseker.
	\newblock Stable vector bundles and the {Frobenius} morphism.
	\newblock {\em Ann. Sci. {\'E}c. Norm. Sup{\'e}r. (4)}, 6:95--101, 1973.
	
	\bibitem[Gie79]{Gieseker79}
	D.~Gieseker.
	\newblock On a theorem of {Bogomolov} on {Chern} classes of stable bundles.
	\newblock {\em Am. J. Math.}, 101:77--85, 1979.
	
	\bibitem[Gro16]{Groechenig16}
	Michael Groechenig.
	\newblock Moduli of flat connections in positive characteristic.
	\newblock {\em Math. Res. Lett.}, 23(4):989--1047, 2016.
	
	\bibitem[Har66]{Hart66}
	R.~Hartshorne.
	\newblock Ample vector bundles.
	\newblock {\em Publ. Math., Inst. Hautes {\'E}tud. Sci.}, 29:63--94, 1966.
	
	\bibitem[Har77]{Hart77}
	Robin Hartshorne.
	\newblock {\em Algebraic geometry}, volume~52 of {\em Graduate Texts in
		Mathematics}.
	\newblock Springer-Verlag New York, 1977.
	
	\bibitem[Hei08]{Hein08}
	Jochen Heinloth.
	\newblock Semistable reduction for {{\(G\)}}-bundles on curves.
	\newblock {\em J. Algebr. Geom.}, 17(1):167--183, 2008.
	
	\bibitem[Hei15]{Hein15}
	Jochen Heinloth.
	\newblock A conjecture of {H}ausel on the moduli space of {H}iggs bundles on a
	curve.
	\newblock {\em Ast\'{e}risque}, (370):157--175, 2015.
	
	\bibitem[Hit87]{Hit87}
	N.J. Hitchin.
	\newblock The self-duality equations on a {Riemann} surface.
	\newblock {\em Proc. London Math. Soc.}, 55(3):39--126, 1987.
	
	\bibitem[HL10]{HL10}
	Daniel Huybrechts and Manfred Lehn.
	\newblock {\em The geometry of moduli spaces of sheaves}.
	\newblock Cambridge University Press, 2010.
	
	\bibitem[JRXY06]{JRXY06}
	Kirti Joshi, S.~Ramanan, Eugene~Z. Xia, and Jiu-Kang Yu.
	\newblock On vector bundles destabilized by {Frobenius} pull-back.
	\newblock {\em Compos. Math.}, 142(3):616--630, 2006.
	
	\bibitem[Kat70]{Ka70}
	Nicholas~M. Katz.
	\newblock Nilpotent connections and the monodromy theorem: Applications of a
	result of {Turrittin}.
	\newblock {\em Publications Math\'ematiques de l'IH\'ES}, (39):175--232, 1970.
	
	\bibitem[Lan04a]{Lan04mixed}
	Adrian Langer.
	\newblock Moduli spaces of sheaves in mixed characteristic.
	\newblock {\em Duke Math. J.}, 124(3):571--586, 2004.
	
	\bibitem[Lan04b]{Lan04positive}
	Adrian Langer.
	\newblock Semistable sheaves in positive characteristic.
	\newblock {\em Ann. of Math. (2)}, 159(1):251--276, 2004.
	
	\bibitem[Lan08]{Lan08}
	Adrian Langer.
	\newblock Lectures on torsion-free sheaves and their moduli.
	\newblock In {\em Algebraic cycles, sheaves, shtukas, and moduli. Impanga
		lecture notes. Based on seminars and schools of Impanga in the period
		2005--2007. A tribute to Hoene-Wro\'nski}, pages 69--103. Basel:
	Birkh{\"a}user, 2008.
	
	\bibitem[Lan09]{Lan09}
	Adrian Langer.
	\newblock Moduli spaces of sheaves and principal {{\(G\)}}-bundles.
	\newblock In {\em Algebraic geometry, Seattle 2005. Proceedings of the 2005
		Summer Research Institute, Seattle, WA, USA, July 25--August 12, 2005}, pages
	273--308. Providence, RI: American Mathematical Society (AMS), 2009.
	
	\bibitem[Lan14]{Lan14}
	Adrian Langer.
	\newblock Semistable modules over {L}ie algebroids in positive characteristic.
	\newblock {\em Doc. Math.}, 19:509--540, 2014.
	
	\bibitem[Lan22]{Lan22}
	Adrian Langer.
	\newblock Moduli spaces of semistable modules over lie algebroids.
	\newblock Preprint, {arXiv}:2107.03128 [math.{AG}] (2022),v2, 2022.
	
	\bibitem[Lau88]{Lau88}
	G\'erard Laumon.
	\newblock Un analogue global du c\^one nilpotent.
	\newblock {\em Duke Mathematical Journal}, 57(2):647--671, 1988.
	
	\bibitem[Li14]{Li14}
	Lingguang Li.
	\newblock The morphism induced by {Frobenius} push-forward.
	\newblock {\em Sci. China, Math.}, 57(1):61--67, 2014.
	
	\bibitem[Li19a]{Li19T}
	Lingguang Li.
	\newblock Frobenius stratification of moduli spaces of rank 3 vector bundles in
	positive characteristic 3. {I}.
	\newblock {\em Trans. Am. Math. Soc.}, 372(8):5693--5711, 2019.
	
	\bibitem[Li19b]{Li19B}
	Lingguang Li.
	\newblock On maximally {Frobenius} destabilised vector bundles.
	\newblock {\em Bull. Aust. Math. Soc.}, 99(2):195--202, 2019.
	
	\bibitem[Li20]{Li20}
	Lingguang Li.
	\newblock Frobenius stratification of moduli spaces of rank 3 vector bundles in
	positive characteristic 3. {II}.
	\newblock {\em Math. Res. Lett.}, 27(2):501--522, 2020.
	
	\bibitem[LP96]{LeP96}
	Joseph Le~Potier.
	\newblock Moduli space of semi-stable fibre bundles and theta functions.
	\newblock In {\em Moduli of vector bundles. Papers of the 35th Taniguchi
		symposium, Sanda, Japan and a symposium held in Kyoto, Japan, 1994}, pages
	83--101. New York, NY: Marcel Dekker, 1996.
	
	\bibitem[LP01]{LP01}
	Yves Laszlo and Christian Pauly.
	\newblock On the {H}itchin morphism in positive characteristic.
	\newblock {\em Internat. Math. Res. Notices}, (3):129--143, 2001.
	
	\bibitem[LP02]{LP02}
	Y.~Laszlo and C.~Pauly.
	\newblock The action of the {Frobenius} map on rank 2 vector bundles in
	characteristic 2.
	\newblock {\em J. Algebr. Geom.}, 11(2):219--243, 2002.
	
	\bibitem[LP08]{LangePauly08}
	Herbert Lange and Christian Pauly.
	\newblock On {Frobenius}-destabilized rank-2 vector bundles over curves.
	\newblock {\em Comment. Math. Helv.}, 83(1):179--209, 2008.
	
	\bibitem[LZ24]{LiZhang24}
	Lingguang Li and Hongyi Zhang.
	\newblock On {Frobenius} stratification of moduli spaces of rank 4 vector
	bundles.
	\newblock {\em Manuscr. Math.}, 173(3-4):961--976, 2024.
	
	\bibitem[MP07]{MP07}
	Vikram~B. Mehta and Christian Pauly.
	\newblock Semistability of {Frobenius} direct images over curves.
	\newblock {\em Bull. Soc. Math. Fr.}, 135(1):105--117, 2007.
	
	\bibitem[MS95]{MS95}
	V.~B. Mehta and S.~Subramanian.
	\newblock Nef line bundles which are not ample.
	\newblock {\em Math. Z.}, 219(2):235--244, 1995.
	
	\bibitem[Nee05]{Ne05}
	Amnon Neeman.
	\newblock {\em The K-Theory of Triangulated Categories}, pages 1011--1078.
	\newblock Springer Berlin Heidelberg, Berlin, Heidelberg, 2005.
	
	\bibitem[New12]{Newstead2012}
	Peter~E. Newstead.
	\newblock {\em Introduction to Moduli Problems and Orbit Spaces}, volume~51 of
	{\em Tata Institute of Fundamental Research Series, Lectures on mathematics}.
	\newblock Tata Institute of Fundamental Research, 2012.
	
	\bibitem[Nit91]{N91}
	Nitin Nitsure.
	\newblock Moduli space of semistable pairs on a curve.
	\newblock {\em Proceedings of The London Mathematical Society}, 62(3):275--300,
	1991.
	
	\bibitem[Nit11]{Nitsure11}
	Nitin Nitsure.
	\newblock Schematic {Harder}-{Narasimhan} stratification.
	\newblock {\em Int. J. Math.}, 22(10):1365--1373, 2011.
	
	\bibitem[Oss06]{Os06}
	Brian Osserman.
	\newblock The generalized {V}erschiebung map for curves of genus 2.
	\newblock {\em Math. Ann.}, 336(4):963--986, 2006.
	
	\bibitem[Oss08]{Os08}
	Brian Osserman.
	\newblock Frobenius-unstable bundles and {{\(p\)}}-curvature.
	\newblock {\em Trans. Am. Math. Soc.}, 360(1):273--305, 2008.
	
	\bibitem[Pau07]{Pauly07}
	Christian Pauly.
	\newblock A smooth counterexample to {Nori}'s conjecture on the fundamental
	group scheme.
	\newblock {\em Proc. Am. Math. Soc.}, 135(9):2707--2711, 2007.
	
	\bibitem[Ram78]{Ram78}
	S.~Ramanan.
	\newblock A note on {C}. {P}. {R}amanujam.
	\newblock In {\em C. {P}. {R}amanujam---a tribute}, volume~8 of {\em Tata Inst.
		Fundam. Res. Stud. Math.}, pages 11--13. Springer, Berlin-New York, 1978.
	
	\bibitem[Sch16]{Schiffmann16}
	Olivier Schiffmann.
	\newblock Indecomposable vector bundles and stable {Higgs} bundles over smooth
	projective curves.
	\newblock {\em Ann. Math. (2)}, 183(1):297--362, 2016.
	
	\bibitem[Ses82]{Seshadri82}
	C.~S. Seshadri.
	\newblock {\em Fibres vectoriels sur les courbes alg{\'e}briques}, volume~96 of
	{\em Ast{\'e}risque}.
	\newblock Soci{\'e}t{\'e} Math{\'e}matique de France (SMF), Paris, 1982.
	
	\bibitem[{Sim}87]{Sim87}
	Carlos~T. {Simpson}.
	\newblock {\em Systems of {Hodge} bundles and uniformization}.
	\newblock PhD thesis, Harvard Univ., 1987.
	
	\bibitem[{Sim}92]{Sim92}
	Carlos~T. {Simpson}.
	\newblock Higgs bundles and local systems.
	\newblock {\em Publications Math\'ematiques de l'IH\'ES}, 75:5--95, 1992.
	
	\bibitem[Sim94]{Sim94II}
	Carlos~T. Simpson.
	\newblock Moduli of representations of the fundamental group of a smooth
	projective variety. {II}.
	\newblock {\em Inst. Hautes {\'E}tudes Sci. Publ. Math.}, (80):5--79, 1994.
	
	\bibitem[Sim98]{Sim98}
	Carlos Simpson.
	\newblock The {H}odge filtration on nonabelian cohomology.
	\newblock In {\em Algebraic geometry--{S}anta {C}ruz 1995}, volume~62 of {\em
		Proc. Sympos. Pure Math.}, pages 217--281. Amer. Math. Soc., Providence, RI,
	1998.
	
	\bibitem[Sim10]{Sim10}
	Carlos Simpson.
	\newblock Iterated destabilizing modifications for vector bundles with
	connection.
	\newblock In {\em Vector bundles and complex geometry}, volume 522 of {\em
		Contemp. Math.}, pages 183--206. Amer. Math. Soc., Providence, RI, 2010.
	
	\bibitem[{Sta}25]{stacks-project}
	The {Stacks project authors}.
	\newblock The stacks project.
	\newblock \url{https://stacks.math.columbia.edu}, 2025.
	
	\bibitem[Sun08]{Sun08}
	Xiaotao Sun.
	\newblock Direct images of bundles under {Frobenius} morphism.
	\newblock {\em Invent. Math.}, 173(2):427--447, 2008.
	
	\bibitem[Vak25]{Vakil25}
	Ravi Vakil.
	\newblock {\em The rising sea---foundations of algebraic geometry}.
	\newblock Princeton University Press, Princeton, NJ, [2025] \copyright 2025.
	
\end{thebibliography}
\end{document}